\documentclass[11pt]{amsart}
\usepackage[lmargin=1in,rmargin=1in,tmargin=1in,bmargin=1in]{geometry}
\usepackage{amsmath, amsthm, amssymb, amsfonts,calrsfs}
\usepackage{enumitem}
\usepackage{xcolor,array}
\usepackage{hyperref}
\usepackage{cleveref}

\hypersetup{
	colorlinks=true,
	linkcolor=blue,
	filecolor=blue,
	citecolor = blue,      
	urlcolor=blue,
}
\usepackage{tikz-cd}
\usepackage{comment}
\usepackage{listings} 

\theoremstyle{plain}
\newtheorem{theorem}{Theorem}[section]
\newtheorem{corollary}[theorem]{Corollary}
\newtheorem{lemma}[theorem]{Lemma}
\newtheorem{prop}[theorem]{Proposition}

\theoremstyle{definition}
\newtheorem{definition}[theorem]{Definition}
\newtheorem{example}[theorem]{Example}
\newtheorem{remark}[theorem]{Remark}

\numberwithin{equation}{section}

\newcommand{\frakp}{\mathfrak p}
\newcommand{\fm}{\mathfrak m}
\newcommand{\fn}{\mathfrak n}

\newcommand{\reg}{\operatorname*{reg}}
\newcommand{\depth}{\operatorname*{depth}}

\newcommand{\supp}{\operatorname{supp}}
\newcommand{\Supp}{\operatorname{Supp}}
\newcommand{\Tor}{\operatorname{Tor}}

\newcommand{\lcm}{\operatorname{lcm}}

\title{Componentwise Linear Ideals From Sums}

\author{Hailong Dao}  
\address{Department of Mathematics, University of Kansas, Lawrence, KS 66045-7523, USA}
\email{hdao@ku.edu}

\author{Sreehari Suresh-Babu}
\address{Department of Mathematics, University of Kansas, Lawrence, KS 66045-7523, USA}
\email{sreehari@ku.edu}

\begin{document}
\begin{abstract}
Let $I,J$ be componentwise linear ideals in a polynomial ring $S$. We study necessary and sufficient conditions for $I+J$ to be componentwise linear. We provide a complete characterization when $\dim S=2$. As a consequence, any componentwise linear monomial ideal in $k[x,y]$ has linear quotients using generators in non-decreasing degrees. In any dimension, we show that under mild compatibility conditions, one can build a componentwise linear ideal from a given collection of componentwise linear monomial ideals using only sum and product with square-free monomials. We provide numerous examples to demonstrate the optimality of our results.  
\end{abstract}
\dedicatory{Dedicated to the memories of J\"urgen Herzog}
\maketitle
\section{Introduction}\label{sec:intro}
Let $I$ be a homogeneous ideal in a polynomial ring $S$. If $I$ is generated in a single degree $d$, recall that $I$ is said to have linear resolution if all the entries in the matrices of the graded minimal resolution of $I$ are $0$ or of degree $1$. Equivalently, the Castelnuovo-Mumford regularity of $I$ is $d$. We say $I$ is componentwise linear if $I_{\langle j\rangle}$ has  linear resolution for all $j$, where $I_{\langle j\rangle}$ denotes the ideal generated by all forms of degree $j$ in $I$.

Introduced by Herzog and Hibi in \cite{HerzogHibi}, componentwise linear ideals have become a very active area of research. Many interesting classes of ideals are componentwise linear: ideals with linear quotients \cite{SharifanVarbaro}, weakly polymatroidal ideals \cite{FVT, MM}, and stable ideals. These ideals also carry remarkable connections to other objects of interest: completely $\fm$-full ideals \cite{HarimaWatanabe}, Koszul ideals \cite{HerzogIyengar, IyengarRomer} and sequentially Cohen-Macaulay simplicial complexes \cite{HRW}. For more comprehensive background and references, we recommend the recent survey \cite{HaTuyl}.

In this paper, we study the following question: when is a sum of ideals componentwise linear? Since adding two ideals is in many ways the simplest ideal operation around, it is highly desirable to have meaningful answers to this question. Clearly, serious assumptions are needed even when we add ideals with linear resolutions: the sum of $(x^2)$ and $(y^2)$ fails to have linear resolution.  

We begin by recalling a few definitions and standard results related to componentwise linearity in \Cref{section:prelims}. In \Cref{sec:sum}, we establish necessary and sufficient conditions for sum of two ideals to be componentwise linear. These results are inspired, and can be viewed as extensions, of many well-known  concepts in literature: ideals with linear quotients (which can be viewed as the precise condition for $I+(f)$ to be componentwise linear), the theory of $\fm$-full ideals, and a result by Hop \cite[Theorem 3.5]{Hop} that tracks the componentwise linearity along short exact sequences.

Our main result in \Cref{sec:sum} is \Cref{thm: maincwl}. Roughly speaking, it says that a system of componentwise linear ideals, under some mild compatibility conditions, can be added up to a componentwise linear ideal. The precise statement is as follows. We say that a collection of square-free monomials $\mathcal L$ is {\it full} if it is closed under taking least common multiple and if monomials $f,g\in \mathcal L$ with $f$ divides $g$, there is a subset  of variables $y_1,\dots, y_s$ such that $fy_1,fy_1y_2,\dots,fy_1...y_s=g$ are in $\mathcal L$. For instance, the sets $\{x,xy\}$ and $\{x,y,xy,xyz\}$ are full. 

\begin{theorem}\label{in_thm1}
    Let $\mathcal L$ be a full set of non-unit squarefree monomials in $S$. To each $f\in \mathcal{L}$, we assign a componentwise linear monomial ideal $I_f$ such that
    \begin{enumerate}
        \item The generators of $I_f$ involve only variables appearing in $f$.
        \item For any $f\in \mathcal{L}$ and any variable $z$ such that $zf\in\mathcal{L}$, we have $I_f\subseteq \mathfrak m I_{zf}$.
    \end{enumerate}
    Then $\sum_{f\in \mathcal{L}} fI_f$ is componentwise linear.
\end{theorem}
We then give a number of examples to show that the assumptions we make are optimal. 

In the last section, we study the case of two variables, which is already quite subtle, in detail. We are able to give a complete and satisfying set of conditions for when the sum of two (not necessarily monomial) componentwise linear ideals to be componentwise linear. Below, $o(I)$ denotes the initial degree of the ideal $I$. 

\begin{theorem}[{\Cref{thm:fullsum}}]\label{in_thm2}
    Let $I,J\subseteq k[x,y]$ be componentwise linear  ideals. Then $I+J$ is componentwise linear  if and only if either of the following holds:
    \begin{enumerate}
        \item $o(I\cap J)=\max\{o(I),o(J)\}$;
        \item $o(I\cap J)=\max\{o(I),o(J)\}+1$ and $(I+J):\fm\neq I:\fm+ J:\fm$.
    \end{enumerate}
\end{theorem}

As a consequence, we give a quick proof that in two variables, any componentwise linear monomial ideal has linear quotients, in a strong sense:

\begin{theorem}[{\Cref{cor_lq}}]\label{in_thm3}
    Let $I$ be a componentwise linear monomial ideal in $k[x,y]$. There exists an ordering of the minimal set of monomial generators $\{f_1,\ldots,f_s\}$ of $I$ such that $\deg f_1\leq \deg f_2\leq\cdots\leq \deg f_s$ and $(f_1,\ldots,f_j):f_{j+1}=(z_j)$ for all $1\leq j\leq s-1$, where each $z_j$ is a variable. 
\end{theorem}

\section{Generalities}\label{section:prelims}
Let $k$ be a field of any characteristic, $S=k[x_1,\ldots,x_n]$ be the standard graded polynomial ring in $n$ variables over $k$, and $\fm=(x_1,\ldots,x_n)$ be the homogeneous maximal ideal. For convenience, such as when dealing with the concept of $\fm$-fullness, we will assume in what follows that $k$ is infinite. However, this condition can be dropped for most of our main applications, since assumptions and conclusions of those are preserved under ground field extensions.

All ideals are assumed to be homogeneous. If $I$ is a homogeneous ideal, write $\mu(I)$ for the minimal number of generators of $I$, and $o(I)$ for the order of $I$, i.e., the smallest degree of nonzero elements in $I$. If $M$ is a finitely generated graded $S$-module, we denote the Castelnuovo-Mumford regularity of $M$ by $\reg M$ and the length of $M$ by $\ell(M)$.

For a monomial ideal $I$, we let $G(I)$ the minimal system of monomial generators of $I$. The support of $I$ is $\Supp(I)=\cup_{f\in G(I)} \supp(f)$, where the support of a monomial $x^a=x_1^{a_1}\cdots x_n^{a_n}$ is $\supp(x^a)=\{x_i\mid a_i> 0\}$. 

We now collect a number of definitions and results related to componentwise linear ideals.
\begin{definition}[\cite{Watanabe}]
    A homogeneous ideal $I\subseteq S$ is said to be \emph{$\fm$-full} if there exists a linear form $x$ such that $\fm I:x=I$.
\end{definition}
\begin{definition}[\cite{HunekeSwanson}]
     A homogeneous ideal $I\subseteq S$ is said to be \emph{full} if there exists a linear form $x$ such that $I:x=I:\fm$.
\end{definition}
\begin{remark}\label{rmk:2dfull}
    If $I$ is $\fm$-full, then $\fm I: x=I$ for a general linear form $x$. Moreover, if $I$ is $\fm$-full, then it is full. The converse holds in $k[x,y]$ \cite[Proposition 14.1.6]{HunekeSwanson}.
\end{remark}
\begin{definition}[\cite{HarimaWatanabe}]
	A homogeneous ideal $I\subseteq S$ is called \emph{completely $\fm$-full} 
	\begin{enumerate}
		\item if $n=0$, then $I=0$;
		\item if $n>0$, then $\fm I:x=I$ and $(I,x)/(x)$ is completely $\fm$-full as an ideal of $S/(x)$, where $x$ is a general linear form.
	\end{enumerate}
\end{definition}
Let $I_{\langle j\rangle}$ denote the ideal generated by all degree $j$ forms in $I$.
\begin{lemma}[{\cite[Theorem 2.3]{HaTuyl}}]\label{lem: high deg}
    Let $I\subseteq S$ be a homogeneous ideal. Then $I_{\langle j\rangle}$ has a linear resolution for all $j\geq \reg I$.
\end{lemma}
\begin{definition}[\cite{HerzogHibi}]
    We say a homogeneous ideal $I\subseteq S$ is \emph{componentwise linear} if $I_{\langle j\rangle}$ has a linear resolution for all $j\geq 0$.
\end{definition}
The main theorem in \cite[Theorem 1.1]{HarimaWatanabe} is that an ideal $I\subseteq S$ is completely $\fm$-full if and only if it is componentwise linear.
\begin{definition}[{\cite[Chapter 8]{MonIdeals}}]
    We say a homogeneous ideal $I\subseteq S$ has \emph{linear quotients} if there exists a minimal system of homogeneous generators $f_1,\ldots,f_s$ such that $(f_1,\ldots,f_{i-1}):f_i$ is generated by linear forms for all $i$.
\end{definition}

\begin{lemma}[{\cite[Lemma 3.1]{HoaTam}}]\label{lemma:reg}
    Let $0\to M\to N\to P\to 0$ be a short exact sequence of graded $S$-modules. Then $\reg P\leq\max\{\reg N, \reg M-1\}$ with equality if $\reg M\neq \reg N$ and $\reg M\leq \max\{\reg N, \reg P+1\}$ with equality if $\reg N\neq \reg P$.
\end{lemma}

\section{Componentwise linearity of certain sums}\label{sec:sum}
We first write down a few necessary conditions on the regularity of $I\cap J$ if $I+J$ is componentwise linear.
\begin{lemma}\label{lem:regcwl1}
    Let $I,J\subseteq S$ be homogeneous ideals such that $I+J$ is componentwise linear. Then 
    \[
    \reg (I\cap J)\leq \max\{\reg I, \reg J\}+1.
    \]
\end{lemma}
\begin{proof}
    We always have $\reg(I\cap J)\leq \max\{\reg I, \reg J, \reg(I+J)+1\}$. Since $I+J$ is a componentwise linear ideal, $\reg(I+J)$ is the maximum degree among the degrees of a minimal system of generators of $I+J$, so we get the desired inequality.
\end{proof}
\begin{prop}\label{prop:regcwl2}
    Let $I,J\subseteq S$ be homogeneous ideals and suppose that $I+J$ is a componentwise linear ideal. Then for a general linear form $x$, we have
    \[
    \reg\left(\frac{S}{(I\cap J, x)}\right)\leq
    \max\left\{\reg\left(\frac{S}{(I, x)}\right), \reg\left(\frac{S}{(J, x)}\right),\reg\left(\frac{S}{I+J}\right)+1\right\}.
    \]
\end{prop}
\begin{proof}
     Consider the short exact sequence
    \[
    0\to \frac{S}{I\cap J}\to \frac SI\oplus \frac SJ\to \frac S{I+J}\to 0.
    \]
    Apply $-\otimes S/(x)=\overline{(-)}$ to the above to obtain
    \begin{equation*}
        \cdots\to \frac{(I+J):x}{I+J}(-1)\to \overline{ \frac{S}{I\cap J}}\to \overline{\frac SI}\oplus \overline{\frac SJ}\to \overline{\frac S{I+J}}\to 0.
    \end{equation*}
    Since $I+J$ is componentwise linear, it is completely $\fm$-full, and in particular satisfies $(I+J):x=(I+J):\fm$. Break the above long exact sequence into two short exact sequences:
    \begin{equation}\label{eq:K}
        0\to K\to \overline{\frac SI}\oplus \overline{\frac SJ}\to \overline{\frac S{I+J}}\to 0
    \end{equation}
    and 
    \begin{equation}\label{eq:K1}
        0\to K_1(-1)\to \overline{ \frac{S}{I\cap J}}\to K\to 0.
    \end{equation}
    Note that $K_1$ is a $k$-vector subspace of $(I+J):\fm/(I+J)$, so $\reg K_1\leq \reg ((I+J):\fm/(I+J))\leq \reg (S/(I+J))$.
    \Cref{eq:K} gives
    \[
    \reg (K)\leq\max\{\reg (S/(I,x)), \reg(S/(J,x)),\reg(S/(I+J,x))+1\},
    \]
    and \Cref{eq:K1} gives
    \[
    \reg\left(\overline{ \frac{S}{I\cap J}}\right)\leq \max\{\reg K_1+1, \reg K\}.
    \]
    Combining the above two inequalities and noting $\reg(S/(I+J,x))\leq \reg (S/(I+J))$, we obtain the desired inequality.
\end{proof}
\begin{remark}
    If a linear form $x$ is a nonzerodivisor on $S/I$, then $\reg(S/I)=\reg(S/(I,x))$. But if $\depth S/I=0$, then, in general, we only have $\reg(S/(I,x))\leq \reg (S/I)$ for a general linear form $x$. 

    We now compare \Cref{lem:regcwl1} and \Cref{prop:regcwl2}. Let $S=k[a,b], I=(a^2,ab,b^m)$ and $J=(a^m,ab,b^2)$ for some $m\geq2$. Then $I+J=(a^2,ab,b^2)$ and $I\cap J=(a^m, ab, b^m)$. We note that $I+J$ is componentwise linear and \Cref{lem:regcwl1} yields
    \[
    m=\reg(I\cap J)\leq \max\{\reg I, \reg J\}+1=m+1.
    \]
    For a general linear form $x$, $\reg(S/(I,x))=o(I)-1$ for any ideal $I\subseteq S$, so \Cref{prop:regcwl2} gives
    \[
    1\leq \max\{1,1, 2\}=2.
    \]
\end{remark}
Next we collect some partial converses to \Cref{lem:regcwl1}.
\begin{lemma}\label{lem:cwl+linear}
    Let $I\subseteq S$ be a componentwise linear ideal and $J\subseteq S$ be an ideal with linear resolution, and suppose that $\reg J\geq \reg I$. If $\reg(I\cap J)=\reg J+1$, then $I+J$ is componentwise linear.
\end{lemma}
\begin{proof}
Consider the short exact sequence
    \[
    0\to I \cap J\to I\oplus J\to I+J\to 0.
    \]
    We note that $\reg(I\cap J)=\reg J+1\neq\reg J=\max\{\reg I,\reg J\}=\reg (I\oplus J)$, so by \Cref{lemma:reg}, $\reg(I+J)=\reg J$. Since $(I+J)_{\langle t\rangle}=I_{\langle t\rangle}$ for $t<\reg J$, it has a linear resolution. The case $t\geq \reg J$ is treated by \Cref{lem: high deg}.
\end{proof}
\begin{remark}
    It is not possible to drop the assumption $\reg J\geq \reg I$ in \Cref{lem:cwl+linear}. Let $S=k[x,y,z]$, $I=(xyz)(x,y,z)+(xy)(x,y)$ and $J=(yz)(y,z)$. Then $I$ is a componentwise linear ideal and $J$ has a linear resolution. Moreover, $I\cap J=(xyz^2, xy^2z)$ and $\reg I\cap J=\reg J+1$, but $I+J=(xy^2, x^2y, yz^2, y^2z)$ is not a componentwise linear ideal. 
    
    It is also not true that if $I$ and $J$ are componentwise linear with $\reg(I\cap J)=\max\{\reg I,\reg J\}+1$, then $I+J$ is componentwise linear. For example, take componentwise linear ideals $I=(x^4,x^3y^2)$ and $J=(y^4,y^3x^2)$ in $S=k[x,y]$. Then $I\cap J=(x^3y^3)$ and $\reg(I\cap J)=\max\{\reg I,\reg J\}+1$, but $I+J=(x^4,x^3y^2,y^3x^2,y^4)$ is not a componentwise linear ideal.
    
    One can also ask if it is possible to generalize \Cref{lem:cwl+linear} by dropping the restriction on $\reg J$, but adding the condition $\reg(I\cap J)=\max\{\reg I,\reg J\}+1$. The answer is no; see \Cref{eg:regplusone}.
\end{remark}
\begin{lemma}\label{lem:lin+lin}
    Let $I,J\subseteq S$ be ideals with $t$-linear resolutions. Then $I+J$ has $t$-linear resolution if and only if $\reg(I\cap J)\leq t+1$.
\end{lemma}
\begin{proof}
    Suppose $\reg(I\cap J)\leq t+1$. If $\reg(I\cap J)=t+1$, then $I+J$ has a $t$-linear resolution by \Cref{lem:cwl+linear}. If not, then $\reg(I\cap J)=t$, since $o(I\cap J)\geq t$. Consider the long exact sequence of $k$-vector spaces
    \begin{equation}\label{eq:torles}
        \cdots\to \Tor_i(I,k)_{i+j}\oplus \Tor_i(J,k)_{i+j}\to \Tor_i(I+J,k)_{i+j}\to \Tor_{i-1}(I\cap J,k)_{i-1+j+1}\to\cdots.
    \end{equation}
    Note that if $j\leq t-1$, $\Tor_i(I+J,k)_{i+j}=0$, since $I+J$ is generated in degree $t$ or higher. If $j\geq t+1$, then using \Cref{eq:torles} we obtain $\Tor_i(I+J,k)_{i+j}=0$, so $I+J$ has $t$-linear resolution.
\end{proof}
\begin{prop}
    Let $I,J\subseteq S$ be componentwise linear ideals. Then $I+J$ is componentwise linear if and only if $\reg(I_{\langle t\rangle}\cap J_{\langle t\rangle})\leq t+1$ for every $t\geq 0$.
\end{prop}
\begin{proof}
    We have $(I+J)_{\langle t\rangle}= I_{\langle t\rangle}+ J_{\langle t\rangle}$, so we apply \Cref{lem:lin+lin} to the pair of ideals $I_{\langle t\rangle}, J_{\langle t\rangle}$ for each $t\geq 0$.
\end{proof}
\begin{prop}\label{HV}
    Let $I,J\subseteq S$ be componentwise linear ideals. Assume that $I\cap J\subseteq \fm I\cap \fm J$. Then $I+J$ is componentwise linear if and only if $I\cap J$ is componentwise linear and $\fm^{s+1}I\cap\fm^{s+1}J=\fm^{s}(I\cap J)$ for all $s\geq 0$. In fact, $I+J$ is componentwise linear if $I\cap J$ is componentwise linear and for each $s\geq 0$ at least one of the following holds: $(I\cap J)\cap \fm^{s+1}I=J\cap \fm^{s+1}I = \fm^s (I\cap J)$ or $(I\cap J)\cap\fm^{s+1}J=I\cap \fm^{s+1}J= \fm^s (I\cap J)$.
\end{prop}
\begin{proof}
Consider the short exact sequence
    \[
    0\to I\cap J\to I\oplus J\to I+J\to 0.
    \]
    We note that $I\cap J$ embeds into $I\oplus J$ using the map $a\mapsto (a,a)$.
Under the given hypotheses, $I+J$ is componentwise linear if and only $I\cap J$ is componentwise linear and 
\begin{equation}\label{msplus}
    \{(a,a):a\in I\cap J\}\cap \fm^{s+1}(I\oplus J)=\fm^s\{(a,a):a\in I\cap J\}
\end{equation}
for all $s\geq 0$ \cite[Theorem 3.10]{HopVu}. \Cref{msplus} is equivalent to $(I\cap J)\cap\fm^{s+1}I\cap\fm^{s+1}J=\fm^{s}(I\cap J)$.

For the second statement, let $s\geq 0$ and suppose that $J\cap \fm^{s+1}I=\fm^s(I\cap J)$. We claim that $\fm^{s+1}I\cap \fm^{s+1} J=\fm^{s}(I\cap J)$. Indeed,
\begin{align*}
    \fm^{s+1}I\cap \fm^{s+1} J &=[\fm^{s+1}I\cap (I\cap J)]\cap \fm^{s+1} J\\
    &=\fm^{s}(I\cap J)\cap \fm^{s+1}J\\
    &=\fm^{s}(I\cap J).
\end{align*}
The last equality is true because of the assumption that $I\cap J\subseteq \fm J$. Thus, $\fm^{s+1}I\cap\fm^{s+1}J=\fm^s(I\cap J)$ for all $s\geq 0$ and since $I\cap J$ is componentwise linear by assumption, we are done by the first part.
\end{proof}
The following result will be used several times.
\begin{lemma}\label{fIgJ}
    Let $I,J\subseteq S$ be ideals and $f,g$ be nonzero elements in $S$. Then 
    \[
    fI\cap gJ=\lcm(f,g)((I:f')\cap (J:g')),
    \]
    where $f', g'$ are such that $\lcm(f,g)=ff'=gg'$.
\end{lemma}
\begin{proof}
    Note that if $u\in (I:f')$, then $\lcm(f,g)u=ff'u\in fI$, so one containment is clear. Conversely, let $u\in fI\cap gJ$. Let $u=fp=gq$ for $p\in I$ and $q\in J$. Since $u$ is divisible by $\lcm(f,g)$, we can write $u=\lcm(f,g) v$ for some $v\in S$. From $fp=ff' v$, we get $v\in (I:f')$, and from $gq=gg'v$, we get $v\in (J:g')$. 
\end{proof}
\begin{prop}\label{prop:disjoint}
    Let $J\subseteq S$ be a componentwise linear monomial ideal and $\fn\subseteq S$ be a monomial prime ideal such that $\Supp(\fn)\cap \Supp(J)=\emptyset$. Then $\fn J$ is a componentwise linear monomial ideal.
\end{prop}
\begin{proof}
    For definiteness, we write $\fn=(x_1,\ldots,x_m)$ and we induct on $m$, the minimal number of generators of $\fn$. If $m=1$, then $\fn J=x_1J\cong J(-1)$, so $\fn J$ is componentwise linear.

    Now assume $m>2$ and the result is true for all smaller values of $m$. We write $\fn=\mathfrak p+ (x_m)$, where $\mathfrak p=(x_1,\ldots,x_{m-1})$. Then $\fn J=\frakp J+ x_m J$ and by induction, we know that $\frakp J, x_mJ$ are componentwise linear. Since intersection distributes over sum for monomial ideals,
    \begin{align*}
        \frakp J\cap x_mJ&=[x_1J+\cdots+ x_{m-1}J]\cap x_m J\\
        &=(x_1 J\cap x_m J)+ \cdots+ (x_{m-1}J \cap x_m J)\\
        &=x_1x_m J+\cdots+ x_{m-1}x_m J&& (\text{by \Cref{fIgJ}})\\
        &=x_m(\frakp J)\cong \frakp J(-1),
    \end{align*}
    so $\frakp J\cap x_m J$ is componentwise linear ideal. Also note that $\frakp J\cap x_m J=x_m \frakp J\subseteq \fm \frakp J\cap \fm x_m J$. Thus by \Cref{HV}, $\fn J$ is componentwise linear ideal if $x_m\frakp J\cap \fm^{s+1} x_mJ=\fm^s x_m \frakp J$ for all $s\geq 0$. The last equality is true if and only if $\frakp J\cap \fm^{s+1} J=\fm^{s}\frakp J$ for all $s\geq 0$.

    To this end, let $u\in \mathfrak pJ\cap\fm^{s+1}J$, so $u=xf=hg$ for some  monomials $f,g\in J$, $x\in G(\mathfrak p)$ and $h\in G(\mathfrak m^{s+1})$. If $x\mid h$, then $h=xh_1$ for some $h_1\in \mathfrak m^s$, so $u\in\mathfrak m^s\mathfrak pJ$. Otherwise, $x\mid g$, and since the minimal generators of $J$ do not involve variables from $\Supp(\mathfrak p)$, we must have $g\in\mathfrak p J$, and thus $u=hg\in \fm^{s+1}\mathfrak pJ\subseteq \fm^{s}\mathfrak pJ$.
\end{proof}
\begin{remark}
    We remark that the assumption $\Supp(\fn)$ and $\Supp(J)$ are disjoint is essential in \Cref{prop:disjoint}. To see why, let $S=k[a,b,c,d]$, $J=(a^2b,abc,bcd,cd^2)$ and $\mathfrak n=(b,c)$. Then $J$ even has a linear resolution, but $\fn J$ is not componentwise linear \cite[Example 2.1]{ConcaHerzog}. It is tempting to apply \cite[Corollary 2.3]{ConcaHerzog} to $(\mathfrak n J)_{\langle j+1\rangle}=\mathfrak n J_{\langle j\rangle}$ to prove \Cref{prop:disjoint}, but notice that $\Supp(J_{\langle j\rangle})$ may not be disjoint from $\Supp(\fn)$.
\end{remark}
\begin{prop}\label{cwl+cwl}
    Let $I,J\subseteq S$ be componentwise linear monomial ideals satisfying $I\cap J\subseteq\fm I$ and $I\cap J=\mathfrak n J$, where $\mathfrak n$ is a monomial prime ideal such that $\Supp(\mathfrak n)\cap \Supp(J)=\emptyset$. Then $I+J$ is a componentwise linear monomial ideal. 
\end{prop}
\begin{proof}
    We have $I\cap J\subseteq \fm I\cap \fm J$ and $I\cap J=\mathfrak nJ$ is componentwise linear by \Cref{prop:disjoint}, so by \Cref{HV}, it suffices to show that $\mathfrak nJ\cap \fm^{s+1}J=\fm^s\mathfrak nJ$ for all $s\geq 0$. But this is exactly the calculation we did in the proof of \Cref{prop:disjoint}, so we are done.
\end{proof}
\begin{remark}
    In \Cref{cwl+cwl}, the assumption $I\cap J\subseteq \fm I$ is important. As an illustrative example, consider $S=k[x,y]$ and componentwise linear ideals $I=(x^3,x^2y^2, xy^3)$ and $J=(y^3)$. Then $I\cap J=(xy^3)=xJ$, but $I+J=(x^3,x^2y^2, y^3)$ is not componentwise linear.
\end{remark}

\begin{definition}
    Let $\mathcal{L}$ be a set of non-unit squarefree monomials in $S$. We say $\mathcal L$ is \emph{lcm-closed} if for every $f,g\in \mathcal{L}$, we have $\lcm(f,g)\in \mathcal L$. If $f,g\in \mathcal{L}$ and $f\mid g$, then a \emph{full path} in $\mathcal{L}$ from $f$ to $g$ is a collection of elements $g_1,g_2\ldots,g_l$ in $\mathcal{L}$ such that $g_1=f$, $g_l=g$ and $g_j=z_j g_{j-1}$ for all $2\leq j\leq l$, where $z_j$ is some variable in $S$. We say $\mathcal L$ is \emph{full} if it is lcm-closed and a full path exists between any two $f,g\in \mathcal{L}$ such that $f\mid g$.
\end{definition}
\begin{theorem}\label{thm: maincwl}
    Let $\mathcal L$ be a full set of non-unit squarefree monomials in $S$. To each $f\in \mathcal{L}$, we associate a componentwise linear monomial ideal $I_f$ such that
    \begin{enumerate}
        \item $G(I_f)\subseteq k[\supp f]$.
        \item For any $f\in \mathcal{L}$ and any variable $z$ such that $zf\in\mathcal{L}$, we have $I_f\subseteq \mathfrak m I_{zf}$.
    \end{enumerate}
    Then $\displaystyle\sum_{f\in \mathcal{L}} fI_f$ is componentwise linear.
\end{theorem}
\begin{proof}
    We induct on $N$, the number of elements in $\mathcal{L}$. If $N=1$, the statement is clear. Before proceeding, we note that the condition (2) above and fullness of $\mathcal{L}$ imply
    \begin{itemize}
        \item[(2')] if $f,g\in\mathcal L$ and $f\mid g$, then $I_f\subseteq I_g$.
    \end{itemize}
    Let $N>1$ and suppose the statement is true for smaller values of $N$. Pick $f\in \mathcal L$ of minimal degree. Consider $\mathcal{H}=\mathcal{L}\setminus \{f\}$. Note that $\mathcal H$ is a full set, so by induction, the ideal $I=\sum_{g\in \mathcal H} gI_g$ is componentwise linear.

    We compute $I\cap fI_f$. For monomial ideals, intersection distributes over sum, so
    \begin{align*}
        I\cap fI_f
        &=\sum_{g\in\mathcal H} gI_g\cap fI_f\\
        &=\sum_{g\in\mathcal{H}} \lcm(f,g)((I_g:g')\cap (I_f:f')),
    \end{align*}
    where $f',g'$ are squarefree monomials such that $\lcm(f,g)=ff'=gg'$. Note that $g'$ consists of variables not in $\supp(g)$, and since $G(I_g)\subseteq k[\supp g]$, we have $(I_g: g')=I_g$. Similarly, $(I_f: f')=I_f$. Thus,
    \begin{align*}
        I\cap fI_f
        &=\sum_{g\in\mathcal H}\lcm(f,g)(I_g\cap I_f).
    \end{align*}
    If $h=\lcm(f,g)$, then $h\in \mathcal H$ and using (2') and fullness of $\mathcal L$, we have $$h(I_g\cap I_f)\subseteq h(I_h\cap I_f)\subseteq h I_f\subseteq zf I_f$$ for some variable $z$ such that $zf\in\mathcal{H}$.
    Thus we obtain
    \begin{align*}
        I\cap fI_f&= \sum_{zf\in \mathcal H}zfI_f=\mathfrak n fI_f,
    \end{align*}
    where $\mathfrak n=\langle z: zf\in \mathcal H\rangle$ is a monomial prime ideal such that $\Supp(\mathfrak n)\cap \Supp(fI_f)=\emptyset$. Note that for a variable $z\in\mathfrak n$, $zfI_f\subseteq \mathfrak m zf I_{zf}$ by condition (2), so $$I\cap fI_f=\mathfrak n fI_f\subseteq \mathfrak m \sum_{zf\in \mathcal H} zfI_{zf}\subseteq \mathfrak m\sum_{g\in \mathcal H}gI_g=\mathfrak m I,$$ hence $I+ fI_f=\sum_{m\in\mathcal{L}}mI_m$ is componentwise linear by \Cref{cwl+cwl}.
\end{proof}
\begin{corollary}\label{cor:maincwl}
    Let $\mathcal{L}$ be a full set of non-unit squarefree monomials in $S$. If $f\in \mathcal{L}$, let $I_f=\langle\supp f\rangle^{a_f-\deg f}$ for some integer $a_f\ge\deg f$. We further assume that if $f,g\in\mathcal L$ and $f\mid g$, then $a_f\geq a_g$. Then $\displaystyle\sum_{f\in \mathcal{L}} fI_f$ is a componentwise linear ideal.
\end{corollary}
\begin{proof}
    We note that if $zf\in\mathcal L$ for some variable $z$ and $f\in\mathcal L$, then 
    \[
    I_f=\langle \supp f\rangle^{a_f-\deg f}=\langle \supp f\rangle \langle \supp f\rangle^{a_f-\deg f-1}\subseteq \mathfrak m \langle \supp(zf)\rangle^{a_{zf}-\deg f-1}=\fm I_{zf},
    \]
    so we are done by \Cref{thm: maincwl}.
\end{proof}
\begin{example}
    Let $S=k[x,y]$ and $\mathcal L=\{xy, x,y\}$. Then $\mathcal L$ is full and \Cref{cor:maincwl} says that 
    \[
    I=xy(x,y)+ (x)^3+ (y)^3
    \]
    is componentwise linear as $a_{xy}=a_x=a_y=3$. On the other hand,
    \[
    J=xy(x,y)^2+ (x)^3+ (y)^3=(x^3,x^2y^2,y^3)
    \]
    is not a componentwise linear ideal and we note that here $a_{xy}=4$ but $a_x=a_y=3< a_{xy}$. This illustrates the sharpness of \Cref{cor:maincwl} even in the case of two variables.
\end{example}

\section{Two Dimensional Polynomial Rings}\label{sec:2dim}
Throughout this section, let $R=k[x,y]$ be the standard graded polynomial ring in two variables and $\fm=(x,y)$ its homogeneous maximal ideal, unless stated otherwise. We recall from \Cref{section:prelims} that, in $k[x,y]$, being componentwise linear is equivalent to being full. In this section, some results concerning componentwise linear ideals are stated in terms of full ideals to indicate that, with suitable modification in terminology, they work over a two dimensional regular local ring.

A careful analysis of the proof of \cite[Proposition 4.2]{Long1} gives us the following:
\begin{lemma}\label{lem:ord}
    Let $I,J\subseteq R$ be full ideals. Then  \[
    o(I\cap J)=\max\{o(I),o(J)\}+ \ell\left(\frac{(I+J):z}{I:\fm+J:\fm}\right)
    \]
    for a general linear form $z\in R$.
\end{lemma}
\begin{proof}
    Consider the exact sequence
    \[
    0\to \frac{R}{I\cap J}\to \frac RI\oplus \frac RJ\to \frac R{I+J}\to 0.
    \]
    We apply $-\otimes R/(z)=\overline{(-)}$ to obtain
    \begin{equation*}
        \cdots \to\frac{I:\fm}{I}(-1)\oplus \frac{J:\fm}{J}(-1)\to \frac{(I+J):z}{I+J}(-1)\to \overline{ \frac{R}{I\cap J}}\to \overline{\frac RI}\oplus \overline{\frac RJ}\to \overline{\frac R{I+J}}\to 0.
    \end{equation*}
    The image of $$\frac{I:\fm}{I}\oplus \frac{J:\fm}{J}\to \frac{(I+J):z}{I+J}$$ is $\dfrac{I:\fm+J:\fm}{I+J}$, so we obtain the exact sequence
    \[
    0\to \frac{(I+J):z}{I:\fm+ J:\fm}(-1)\to \overline{ \frac{R}{I\cap J}}\to \overline{\frac RI}\oplus \overline{\frac RJ}\to \overline{\frac R{I+J}}\to 0.
    \]
    Finally, we observe that $\ell(\overline{R/I})=o(I)$ and $o(I+J)=\min\{o(I),o(J)\}$.
\end{proof}
\begin{prop}\label{prop:maxlength}
    Let $I,J\subseteq R$ be such that $I+J$ is a full ideal. Then $$0\leq o(I\cap J)-\max\{o(I),o(J)\}\leq 1.$$
\end{prop}
\begin{proof}
    Consider the short exact sequence
    \[
    0\to \frac{R}{I\cap J}\to \frac RI\oplus \frac RJ\to \frac R{I+J}\to 0.
    \]
    Let $z$ be a general linear form and tensor the above sequence with $R/(z)=A$, where $A$ is a discrete valuation ring with parameter $t$. We obtain the long exact sequence
    \[
    \cdots\to\frac{(I+J):z}{I+J}(-1)\to \frac{A}{(t^{o})}\to \frac{A}{(t^{o(I)})}\oplus \frac{A}{(t^{o(J)})}\to\cdots
    \]
    where $o=o(I\cap J)$. Since $I+J$ is full, $(I+J):z=(I+J):\fm$, so the leftmost term is killed by $\fm$, thus the image must live inside the socle of $A/(t^o)$. If the image is zero, then
    \[
    0\to \frac{A}{(t^o)}\to \frac{A}{(t^{o(I)})}\oplus \frac{A}{(t^{o(J)})}\to \frac{A}{(t^{o(I+J)})}\to 0
    \]
    is exact and we obtain $o=\max\{o(I),o(J)\}$. Otherwise, the cokernel is $A/(t^{o-1})$ and we obtain the exact sequence
    \[
    0\to \frac A{(t^{o-1})}\to \frac{A}{(t^{o(I)})}\oplus\frac{A}{(t^{o(J)})}\to \frac{A}{(t^{o(I+J)})}\to 0,
    \]
    thus $o-1=\max\{o(I),o(J)\}$.
\end{proof}
We can now write down a precise condition for the sum of two full ideals to be full.
\begin{theorem}\label{thm:fullsum}
    Let $I,J\subseteq R$ be full ideals. Then $I+J$ is full if and only if either of the following holds:
    \begin{enumerate}
        \item $o(I\cap J)=\max\{o(I),o(J)\}$;
        \item $o(I\cap J)=\max\{o(I),o(J)\}+1$ and $(I+J):\fm\neq I:\fm+ J:\fm$.
    \end{enumerate}
\end{theorem}
\begin{proof}
    Suppose $I+J$ is full. By \Cref{prop:maxlength}, we can assume $o(I\cap J)=\max\{o(I),o(J)\}+1$. Since $(I+J):z=(I+J):\fm$ for a general linear form $z$, Lemma \ref{lem:ord} implies $\ell((I+J):\fm/(I:\fm+J:\fm))=1$, so $(I+J):\fm\neq I:\fm+ J:\fm$.

    Conversely, if (1) holds, then $I+J$ is full by \cite[Corollary 4.3]{Long1}. Next suppose (2) holds. Then using Lemma \ref{lem:ord}, we obtain $\ell((I+J):z/I:\fm+J:\fm)=1$. Since
    \[
    1=\ell\left(\frac{(I+J):z}{I:\fm+ J:\fm}\right)=\ell\left(\frac{(I+J):\fm}{I:\fm+J:\fm}\right)+ \ell\left(\frac{(I+J):z}{(I+J):\fm}\right),
    \]
    we must have $(I+J):z=(I+J):\fm$, i.e., $I+J$ is full.
\end{proof}
\begin{example}
    Consider the full ideals $I=(x^4,x^2y^2,x^3y,y^3)$ and $J=(x^4,x^2y,xy^2,y^4)$. The sum $I+J=(x^4,x^2y,xy^2,y^3)$ is full. Here the order of $I\cap J$ is $4=\max\{o(I),o(J)\}+1$ and $(I+J):\fm=(x^3,xy,y^2)$ and $I:\fm+ J:\fm=(x^3,xy,y^3)$.
\end{example}

\begin{prop}\label{mu(I+J)}
    Let $I,J$ be full ideals and assume $\mu(I+J)=\mu(I)+\mu(J)$. Then $I+J$ is full if and only if $o(I\cap J)\leq \max\{o(I),o(J)\}+1$.
\end{prop}
\begin{proof}
    By \Cref{thm:fullsum}, the only case left to show is when $o(I\cap J)=\max\{o(I),o(J)\}+1$, and it suffices to show $(I+J):\fm\neq (I:\fm)+(J:\fm)$. To this end, let $\ell(I:\fm/I)=s$ and write $(I:\fm)/I=k^s$. Applying $-\otimes R/J$ to the exact sequence
    \[
    0\to k^s\to R/I\to R/(I:\fm)\to 0
    \]
    gives us
    \[
    \cdots\to k^s\to \frac{R}{I+J}\to \frac R{(I:\fm)+J}\to 0.
    \]
    Thus $k^s$ surjects onto $(I:\fm)+J/(I+J)$, and hence
    $$s=\mu(I)-1\geq \ell\left(\frac{(I:\fm)+J}{I+J}\right).$$
    Let $a=\mu(I)$ and $b=\mu(J)$. Consider 
    \[
    I+J\subseteq (I:\fm)+J\subseteq (I:\fm)+(J:\fm)\subseteq (I+J):\fm.
    \]
    The first quotient has length at most $a-1$ and by a similar estimate the second has length at most $b-1$. But the length of $(I+J):\fm/(I+J)$ is $\mu(I+J)-1=a+b-1$, so $(I:\fm)+ (J:\fm)$ is strictly contained inside $(I+J):\fm$.
\end{proof}

\begin{prop}\label{linearform}
    Let $I\subseteq R$ be a full ideal and $f$ be a homogeneous element such that $d=\deg f\geq o(I)$ and $\mu(I+(f))=\mu(I)+1$. Then $I+(f)$ is full if and only if $(I:f)$ contains a linear form.
\end{prop}
\begin{proof}
    Since $I+(f)$ is a full ideal, $d\leq o(I\cap (f))=o(f(I:f))\leq d+1$. But $(I:f)\neq (1)$, so we have $o(f(I:f))\geq d+1$. Thus $o(f(I:f))=d+1$ and hence $(I:f)$ must contain a linear form.

    Conversely, suppose $(I:f)$ contains a linear form. Then $o(I\cap (f))=d+1$ and thus $(I,f)$ is full by \Cref{mu(I+J)}.
\end{proof}
\begin{prop}\label{principalNlinear}
    Let $I$ be a full ideal, and $f$ be an element such that $d=\deg f\ge o(I)$. Suppose $I+(f)$ is a full ideal with $\mu(I+(f))=\mu(I)+1$. Then $(I:f)=(z)$ for some linear form $z$.
\end{prop}
\begin{proof}
    
    \Cref{linearform} shows that $(I:f)$ contains a linear form. We note that $I$ is not $\fm$-primary, because otherwise $o(I+(f))=o(I)=\mu(I)-1$ and $\mu(I+(f))-1=\mu(I)$, so $I+(f)$ would not be full.

    Write $I=aJ$, where $J$ is full ideal of finite colength. Let $h=\gcd(a,f)$ and set $a=ha_1$ and $f=hf_1$. We have
    \begin{align*}
        I:f=ha_1J:hf_1=a_1J: f_1= a_1(J:f_1).
    \end{align*}
    Suppose that $a_1$ is a unit. Then $a$ divides $f$. Write $f=af_2$. If $f_2\in J$, then we get $f\in I$, a contradiction. If $f_2\notin J$, then $I+(f)=a(J+(f_2))$ with $J+(f_2)$ an $\fm$-primary full ideal. Then $$\mu(I+(f))=\mu(J+(f_2))=o(J+(f_2))+1=o(J)+1=\mu(J)=\mu(I),$$ again a contradiction. Since the order of $(I:f)$ is one, $a_1$ must be a linear form and $(I:f)=(a_1)$.
\end{proof}

\begin{theorem}\label{thm: ordering}
    Let $I$ be a componentwise linear monomial ideal in $R=k[x,y]$. There exists an ordering of the minimal set of monomial generators $\{f_1,\ldots,f_s\}$ of $I$ such that $\deg f_1\leq \deg f_2\leq\cdots\leq \deg f_s$ and $I_j=(f_1,\ldots,f_j)$ is componentwise linear for all $1\leq j\leq s$.
\end{theorem}
\begin{proof}
    Let $f_1\in G(I)$ be a monomial of degree equals the order of $I$. Clearly, $I_1=(f_1)$ is componentwise linear. Inductively, suppose we have constructed a componentwise linear ideal $I_j=(f_1,\ldots,f_j)$ and write $I=I_j+J$, where $J$ is the ideal generated by the remaining set of minimal generators of $I$. We assume we have constructed $I_j$ in such a way that $\reg I_j\leq o(J)$. Since $I$ is componentwise linear, by \Cref{prop:maxlength}, $o(I_j\cap J)\leq \max\{o(I_j),o(J)\}+1=o(J)+1$. Now $I_j\cap J$ is generated by $L=\{\lcm(f_l, f_m)\mid f_l\in G(I_j), f_m\in G(J)\}$. Since $G(J)\cup G(I_j)=G(I)$ and $G(J)\cap G(I_j)=\emptyset$, each monomial in $L$ has degree at least $o(J)+1$. Thus, $o(I_j\cap J)=o(J)+1$. Choose $f_{j+1}\in G(J)$ of degree $o(J)$ so that $f_{j+1}$ times a variable is in $I_j\cap J$. We note that $o(I_j\cap (f_{j+1}))=o(J)+1$, and $I_j:f_{j+1}$ contains a linear form, so $I_{j+1}=I_j+(f_{j+1})$ is componentwise linear by \Cref{linearform}. The ideal $J'$ generated by the remaining generators satisfies the condition that $\reg I_{j+1}\leq o(J')$, so we can now repeat the argument with $I_{j+1}$ and $J'$.
\end{proof}
\begin{remark}
    For homogeneous ideals, \Cref{thm: ordering} not true. For example, consider the componentwise linear ideal $I=(x^2+y^2, x^3, y^3)$. There is no way to arrange this minimal system of generators so that \Cref{thm: ordering} holds.
\end{remark}
We now show that a componentwise linear monomial ideal in $R$ has linear quotients with respect to an ordering in which generators are arranged in a nondecreasing order of degrees.
\begin{corollary}\label{cor_lq}
    Let $I$ be a componentwise linear monomial ideal in $k[x,y]$. There exists an ordering of the minimal set of monomial generators $\{f_1,\ldots,f_s\}$ of $I$ such that $\deg f_1\leq \deg f_2\leq\cdots\leq \deg f_s$ and $(f_1,\ldots,f_j):f_{j+1}=(z_j)$ for all $1\leq j\leq s-1$, where each $z_j$ is a variable. 
\end{corollary}
\begin{proof}
    We retain the notation from \Cref{thm: ordering}. We claim that $I_{j}:f_{j+1}$ is generated by a variable. Indeed, by \Cref{thm: ordering}, we have $\deg f_{j+1}\geq o(I_j)$ and $I_{j+1}=I_{j}+(f_{j+1})$ is componentwise linear with $\mu(I_{j+1})=\mu(I_j)+1$, so we are done by \Cref{principalNlinear}.
\end{proof}
\begin{example}
    Take the componentwise linear ideal $I=(x^3,xy,y^3)$ and consider the minimal generators in that particular order. Then $I_1=(x^3)$ is componentwise linear, $I_1\cap (xy)=(x^3y)$ and $(I_1+(xy)):\fm=(x^2,xy)\neq (x^3)+(xy)$, so $I_2=I_1+(xy)$ is componentwise linear by \Cref{thm:fullsum}. Next, $I_2\cap (y^3)=(xy^3)$ and $(I_2+(y^3)):\fm=(x^2,xy,y^2)\neq (x^2,xy)+(y^3)$, so $I_3=I_2+(y^3)$ is componentwise linear, again by \Cref{thm:fullsum}. Note that the generators are not in a nondecreasing order of degrees.
\end{example}

\begin{lemma}\label{lem:regndeg}
    Let $S=k[x_1,\ldots,x_n]$ and $\fm$ be its homogeneous maximal ideal. Let $I\subseteq S$ be a homogeneous $\fm$-primary ideal and $f\in S$ be a nonzero homogeneous element. Then $\reg(I:f)+\deg f\leq \reg I$ if and only if $\deg f\leq \reg I$.
\end{lemma}
\begin{proof}
    If $\deg f\leq \reg I$, then either $f\in I$, in which case $(I:f)=S$ and $\reg (I:f)=0$, or $f\notin I$, in which case $(I:f)$ is an $\fm$-primary ideal. Then the inequality follows from the inclusion
    \[
    0\to \frac{S}{(I:f)}(-\deg f)\xrightarrow{f} \frac{S}{I}.
    \]
\end{proof}
\begin{lemma}\label{lem:regm_primary}
    Let $I,J\subseteq S$ be $\fm$-primary ideals. Then $\reg(I\cap J)\leq \max\{\reg I,\reg J\}$.
\end{lemma}
\begin{proof}
    Let $\bar f$ be a nonzero homogeneous socle element of $S/I\cap J$ with $\deg f=\reg(S/(I\cap J))=\reg(I\cap J)-1$. If $\reg (I\cap J)>\max\{\reg I, \reg J\}$, then $f\in I\cap J$, so $\bar f=0$, a contradiction.
\end{proof}
\begin{prop}\label{dim2maxprop}
    Let $I, J$ be homogeneous ideals in $R=k[x,y]$. If $I\cap J$ is not principal, then $\reg(I\cap J)\leq\max\{\reg I,\reg J\}$.
\end{prop}
\begin{proof}
    Write $I=fI'$ and $J=gJ'$, where $I',J'$ are  ideals of finite colength (which could be $R$). If $\lcm(f,g)=ff'=gg'$ for some $f',g'\in R$, then
    \[
    I\cap J=\lcm(f,g)((I':f')\cap (J':g')).
    \]
    Note that $(I':f')=(J':g')=R$ if and only if $I\cap J$ is principal. Thus at least one of them, say $I':f'$, must be a proper ideal, so $\reg(I':f')\leq \reg I'-\deg f'$ by \Cref{lem:regndeg}. Thus
    \begin{align*}
        \reg(I\cap J)&=\deg \lcm(f,g)+ \reg((I':f')\cap (J':g'))\\
        &\leq \deg \lcm(f,g)+\max\{\reg I'-\deg f',\reg J'-\deg g'\}&&(\text{by \Cref{lem:regm_primary}})\\
        &=\max\{\reg I,\reg J\}.
    \end{align*}
\end{proof}

\begin{theorem}
    Let $I, J$ be componentwise linear ideals in $R=k[x,y]$. Suppose that $\reg(I\cap J)=\max\{\reg I,\reg J\}+1$. Write $I=fI', J=gJ'$ with $I',J'$ of finite colength. Let $\lcm(f,g)= ff'=gg'$. The following are equivalent: 
  \begin{enumerate}
    \item $I+J$ is componentwise linear.
    \item $(\fm^{s+1}I':f')\cap (\fm^{s+1}J':g')\subseteq \fm^s$ for all $s\geq 0$. 
  \end{enumerate}
\end{theorem}

\begin{proof}
    By \Cref{dim2maxprop} (and its proof), we see that $I\cap J =\lcm(f,g)$ is principal, or in other words, $f'\in I'$ and $g'\in J'$. The condition on regularity now translates to 
    $$ 0  =\max\{\reg I' -\deg f', \reg J'-\deg g'\}+1.$$
    As $I',J'$ are componentwise linear, their regularity coincide with the maximal degree of generators, so $f'\in \fm I', g'\in \fm J'$. It follows that $I\cap J \subset \fm(I\oplus J)$, and now \Cref{HV}
    shows that (1) is equivalent to 
    $\fm^{s+1}I\cap\fm^{s+1}J=\fm^s(I\cap J)$ for all $s\geq 0$, which can easily seen to be equivalent to (2). 
\end{proof}

\begin{example}\label{eg:regplusone}
    Consider the componentwise linear ideals $I=(x^3)(x^2,xy,y^3)$ and $J=(y^4)(x,y)$. The ideal $J$ has a linear resolution, $I\cap J=(x^3y^4)$ and $\reg(I\cap J)=\max\{\reg I, \reg J\}+1$. But $I+J$ is not componentwise linear. We note that $(\fm^3J':g')=(\fm^4:x^3)=\fm$ and $(\fm^3I':f')=(\fm^3(x^2,xy,y^3): y^4)=(x,y^2)$, so $(\fm^3 I':f')\cap (\fm^3 J': g')\nsubseteq \fm^2$.
\end{example}

\section*{Acknowledgements}
We thank T{\`a}i H{\`a}, Adam Van Tuyl, and Jeff Mermin for helpful conversations. The first author acknowledges partial support from Simons Foundation grant MP-TSM-00002378. Computations done using \texttt{Macaulay2} \cite{M2} have been crucial in our investigation. 

\bibliographystyle{amsalpha}
\bibliography{ref}
\end{document}